\documentclass[english]{article}
\usepackage[T1]{fontenc}
\usepackage[latin9]{inputenc}
\usepackage{units}
\usepackage{mathrsfs}
\usepackage{mathtools}
\usepackage{amsmath}
\usepackage{amsthm}
\usepackage{amssymb}
\usepackage{stmaryrd}
\usepackage{stackrel}
\usepackage{graphicx}
\usepackage[all]{xy}

\makeatletter

\newcommand{\noun}[1]{\textsc{#1}}

\providecommand{\tabularnewline}{\\}

\numberwithin{equation}{section}
\numberwithin{figure}{section}
\theoremstyle{plain}
\newtheorem{thm}{\protect\theoremname}
\theoremstyle{plain}
\newtheorem{prop}[thm]{\protect\propositionname}
\theoremstyle{remark}
\newtheorem{rem}[thm]{\protect\remarkname}
\theoremstyle{plain}
\newtheorem{lem}[thm]{\protect\lemmaname}

\usepackage{babel}
\usepackage{bbold}
\usepackage[all]{xy}

\makeatother

\usepackage{babel}
\providecommand{\lemmaname}{Lemma}
\providecommand{\propositionname}{Proposition}
\providecommand{\remarkname}{Remark}
\providecommand{\theoremname}{Theorem}

\begin{document}
\title{\noun{deformation of the Okubic Albert algebra and its relation to
the Okubic affine and projective planes}}
\author{Daniele Corradetti$^{*}$, Alessio Marrani$^{\dagger}$, Francesco Zucconi$^{\ddagger}$}
\maketitle
\begin{abstract}
We present a deformation of the Okubic Albert algebra introduced by
Elduque whose rank-1 idempotent elements are in biunivocal correspondence
with points of the Okubic projective plane, thus extending to Okubic
algebras the correspondence found by Jordan, Von Neumann, Wigner between
the octonionic projective plane and the Albert algebra $\mathfrak{J}_{3}\left(\mathbb{O}\right)$. 

\medskip{}

\textbf{Msc}: 17A20; 17A35; 17A75; 51A35. 
\end{abstract}

\section{Introduction }

All unital composition algebras, i.e. all Hurwitz algebras such as
$\mathbb{R}$, $\mathbb{C}$, $\mathbb{H}$,$\mathbb{O}$ and all
their split companions, allow the construction of projective lines
and planes which found interesting and useful applications in alternative
formulations of Quantum Mechanics \cite{Jordan}. Inspired by an algebrical
definition of the Spin group, given by Elduque in \cite{ElDuque Comp},
in a previous work two of the authors defined an affine and projective
plane over the Okubo algebra \cite{Corr-OkuboSpin}, noting that,
with little modifications, Veronese coordinates could still be used
even though Okubo algebra is neither alternative nor unital (a review
of this construction is done in sec. 4). Crucial to this construction
is the \emph{flexibility} of the algebra, i.e. for every $x$ and
$y$ in the algebra we have $x*\left(y*x\right)=\left(x*y\right)*x$,
which is a softer requirement than alternativity since all alternative
algebras are also flexible, as famously shown by Artin, but not the
other way around. Flexibility, together with the property of composition
of the norm, defines a \emph{symmetric composition} algebra (whose
basics are here briefly reviewed in sec. 2). 

The real Okubo algebra $\mathcal{O}$, also known as pseudoctonions,
is deeply linked to octonions $\mathbb{O}$ since a deformation of
the Okubic product results in the octonionic multiplication and the
converse is also true \cite{Okubo 1978,Okubo 78c}. In the octonionic
case, but more generally in any Hurwitz algebra, a simple and elegant
relation links rank-1 idempotents of the Albert algebra $\mathfrak{J}_{3}\left(\mathbb{O}\right)$
and points in the projective plane $\mathbb{O}P^{2}$. We find that
a similar relation still exists between rank-1 idempotents of a deformation
of the \emph{Okubic Albert algebra} $\mathbb{A}_{q}\left(\mathcal{O}\right)$
introduced by Elduque in \cite{Elduque Gradings SC} and points of
the Okubic projective plane $\mathcal{O}P^{2}$. While the Okubic
Albert algebra is a Jordan algebra (isomorphic to $\mathfrak{J}_{3}\left(\mathbb{O}_{s}\right)$,
see \cite[Thm 5.15]{Elduque Gradings SC}), its deformation for $q=1/2$
it is not, even though the algebra remains commutative and flexible.
Nevertheless $\mathbb{A}_{1/2}\left(\mathcal{O}\right)$ is, to the
best of our current knowledge, the unique algebra that turns the following
diagram over the point of the affine plane $\mathcal{O}A^{2}$ into
a commutative diagram, i.e.
\begin{equation}
\xymatrix{\underset{\text{points}}{\mathcal{O}A^{2}} & \ar[r] & \underset{\text{Veronese vectors}}{\mathcal{O}P^{2}}\\
 & \ar[rd] & \ar[d]\\
 &  & \underset{\begin{array}{c}
\text{Tr}=1,\\
\text{idempotent elements}
\end{array}}{\mathbb{A}_{1/2}\left(\mathcal{O}\right)}
}
\end{equation}

To be more explicit, we here anticipate that, while it is possible
to modify the Veronese conditions in (\ref{eq:Ver-1-1}) and (\ref{eq:Ver-2-1})
in order to obtain the identification between trace-one idempotents
of $\mathbb{A}_{q}\left(\mathcal{O}\right)$ and Veronese vectors,
it is not possible in general, i.e. for any value of the deformation
parameter $q$, to also maintain the correspondence between the newly
defined Veronese vectors and the points of the Okubic affine plane
$\mathcal{O}A^{2}$ with a modification of (\ref{eq:corrispondenza1}).

In section 2 we present symmetric composition algebras and Okubo algebras,
mainly following \cite{Elduque Gradings SC,ElDuque Comp,ElduQueAut}
but introducing also some new elements. In section 3 we present for
the first time the Okubic projective line showing that is the one
point compactification of the Okubo algebra $\mathcal{O}$. In section
4 we present the Okubic affine and projective plane, following \cite{Corr-OkuboSpin},
and briefly showing the embedding of the affine plane into the projective
plane. In section 5 we introduce the deformed Okubic Albert algebra
$\mathbb{A}_{q}\left(\mathcal{O}\right)$ and show the correspondence
between its trace-one idempotents and the Veronese vectors. Finally,
inspired by \cite{Freud 1965}, we investigate and characterise the
automorphisms of $\mathbb{A}_{q}\left(\mathcal{O}\right)$.

\section{The Okubo algebra}

An \emph{algebra} $A$ is a vector space over a ground field $\mathbb{K}$
with a bilinear multiplication. If $A$ is also endowed with a non-degenerate
quadratic norm $n$ from $A$ to $\mathbb{K}$ such that 
\begin{equation}
n\left(x\cdot y\right)=n\left(x\right)n\left(y\right),
\end{equation}
for every $x,y\in A$, then the algebra is called a \emph{composition
algebra}. If the algebra is also unital, i.e. if it exists an element
$1$ such that $x\cdot1=1\cdot x=x$, then the algebra is an \emph{Hurwitz
algebra}. 

The interplay between the composition property of the norm and the
existence of a unit element is full of interesting implications \cite[Sec. 2]{ElDuque Comp}.
Indeed, every Hurwitz algebra is endowed with an order-two anti-automorphism
called \emph{conjugation}, defined as 
\begin{equation}
\overline{x}\coloneqq\left\langle x,1\right\rangle 1-x,\label{eq:conjugation}
\end{equation}
where $\left\langle x,y\right\rangle $ is the polar form of the norm
which is given by
\begin{equation}
\left\langle x,y\right\rangle =n\left(x+y\right)-n\left(x\right)-n\left(y\right).\label{eq:linearization}
\end{equation}
Moreover, definition (\ref{eq:linearization}), combined with the
composition, yields to the notable relations 
\begin{align}
\left\langle x\cdot y,z\right\rangle  & =\left\langle y,\overline{x}\cdot z\right\rangle ,\\
\left\langle y\cdot x,z\right\rangle  & =\left\langle y,z\cdot\overline{x}\right\rangle ,
\end{align}
 that imply (cfr. \cite[Prop. 2.2]{ElDuque Comp}) on the conjugation
that
\begin{equation}
x\cdot\overline{x}=n\left(x\right)1\,\,\,\,\text{and}\,\,\,\,\,\overline{x\cdot y}=\overline{y}\cdot\overline{x},
\end{equation}
and finally, a relation that is crucial for the consistent definition
of the Veronese coordinates on composition algebras, i.e. 
\begin{equation}
x\cdot\left(\overline{x}\cdot y\right)=\left(x\cdot\overline{x}\right)\cdot y=n\left(x\right)y.
\end{equation}
It is well known that, up to isomorphisms, the only Hurwitz algebras
are $\mathbb{R},\mathbb{C},\mathbb{H}$ and $\mathbb{O}$ along with
their split companions \cite[Cor. 2.12]{ElDuque Comp}. Particularly
relevant for this work are the split octonions $\mathbb{O}_{s}$,
which we will define as the real vector space endowed with the multiplication
table in Tab. \ref{tab:Split Octonions} over the canonical base $\left\{ e_{1},e_{2},u_{1},u_{2},u_{3},v_{1},v_{2},v_{3}\right\} $
obtained starting from two idempotents $e_{1}$ and $e_{2}$ and two
three-dimensional real vector subspaces defined as $V=\left\{ e_{1}\cdot v=0\right\} $
and $U=\left\{ e_{2}\cdot u=0\right\} $.
\begin{center}
\begin{table}
\begin{centering}
\begin{tabular}{|c|cc|ccc|ccc|}
\hline 
 & $e_{1}$ & $e_{2}$ & $u_{1}$ & $u_{2}$ & $u_{3}$ & $v_{1}$ & $v_{2}$ & $v_{3}$\tabularnewline
\hline 
$e_{1}$ & $e_{1}$ & 0 & $u_{1}$ & $u_{2}$ & $u_{3}$ & 0 & 0 & 0\tabularnewline
$e_{2}$ & 0 & $e_{2}$ & 0 & 0 & 0 & $v_{1}$ & $v_{2}$ & $v_{3}$\tabularnewline
\hline 
$u_{1}$ & 0 & $u_{1}$ & 0 & $v_{3}$ & $-v_{2}$ & $-e_{1}$ & 0 & 0\tabularnewline
$u_{2}$ & 0 & $u_{2}$ & $-v_{3}$ & 0 & $v_{1}$ & 0 & $-e_{1}$ & 0\tabularnewline
$u_{3}$ & 0 & $u_{3}$ & $v_{2}$ & $-v_{1}$ & 0 & 0 & 0 & $-e_{1}$\tabularnewline
\hline 
$v_{1}$ & $v_{1}$ & 0 & $-e_{2}$ & 0 & 0 & 0 & $u_{3}$ & $-u_{2}$\tabularnewline
$v_{2}$ & $v_{2}$ & 0 & 0 & $-e_{2}$ & 0 & $-u_{3}$ & 0 & $u_{1}$\tabularnewline
$v_{3}$ & $v_{3}$ & 0 & 0 & 0 & $-e_{2}$ & $u_{2}$ & $-u_{1}$ & 0\tabularnewline
\hline 
\end{tabular}
\par\end{centering}
\medskip{}

\caption{\label{tab:Split Octonions}Multiplication table of the split octonions
in the canonical base, from \cite{ElDuque Comp}.}

\end{table}
\par\end{center}

\subsection*{Symmetric composition algebras}

Starting from any Hurwitz algebra $\left(A,\cdot,n\right)$ with conjugation
(\ref{eq:conjugation}) we define a new multiplication 
\begin{equation}
x\circ y\coloneqq\overline{x}\cdot\overline{y},
\end{equation}
that, again, is composition with respect to the norm, i.e. $n\left(x\circ y\right)=n\left(x\right)n\left(y\right)$.
This allows to define another algebra $\left(A,\circ,n\right)$, called\emph{
para-Hurwitz} algebra \cite{OkMy80}. Since $1\circ x=x\circ1=\overline{x}$,
then para-Hurwitz algebras are not unital (thus their existence do
not contradict Hurwitz theorem \cite{ElDuque Comp}); nevertheless,
they are composition and flexible and therefore \emph{symmetric composition}
\cite{KMRT}, i.e. they satisfy the following relation
\begin{equation}
\left\langle x\cdot y,z\right\rangle =\left\langle x,y\cdot z\right\rangle ,
\end{equation}
 or, equivalently, the following one
\begin{equation}
x\circ\left(y\circ x\right)=\left(x\circ y\right)\circ x=n\left(x\right)y.
\end{equation}

A similar approach was already used by Petersson \cite{Petersson 1969}
who used, instead of an anti-homomorphism of order two, an homomorphism
of order three $\tau$ that he used to deform the product of the Hurwitz
algebras $\left(A,\cdot,n\right)$ as follows
\begin{equation}
x*y\coloneqq\tau\left(\overline{x}\right)\cdot\tau^{2}\left(\overline{y}\right).
\end{equation}
Again, the algebra $\left(A,*,n\right)$ turned out to be a symmetric
composition algebra, i.e. 
\begin{align}
n\left(x*y\right) & =n\left(x\right)n\left(y\right),\\
x*\left(y*x\right) & =\left(x*y\right)*x=n\left(x\right)y.
\end{align}
The interesting case is here given when the starting Hurwitz algebra
is the one of split octonions $\mathbb{O}_{s}$ and the order-three
homomorphism is defined as
\begin{equation}
\tau\left(e_{i}\right)\coloneqq e_{i},\,\,\,\tau\left(u_{i}\right)\coloneqq u_{i+1},\,\,\,\,\tau\left(v_{i}\right)\coloneqq v_{i+1},
\end{equation}
 where indices are considered modulo $3$ and $\left\{ e_{1},e_{2},u_{1},u_{2},u_{3},v_{1},v_{2},v_{3}\right\} $
is the canonical base of $\mathbb{O}_{s}$. In this specific case
the construction yields to a non-unital symmetric composition algebra
that is not obtainable through the use of the previous order-two deformation,
and that was indipendently discovered by Okubo\cite{Okubo 78c}; nowadays,
this algebra is named \emph{Okubo algebra} , and denoted by $\mathcal{O}$
\cite{Elduque Gradings SC}.

\subsection*{Okubo and split-Okubo algebras }

Let $\eta_{1}$ be $\text{diag}\left(1,1,1\right)$ while $\eta_{2}$
be $\text{diag}\left(-1,1,1\right)$. Then, consider the real vector
spaces of three by three traceless matrices over the complex numbers
$\mathbb{C}$ such that 
\begin{equation}
x^{\dagger}\coloneqq\eta_{i}x\eta_{i},
\end{equation}
for $i=1,2$, where $x^{\dagger}$ is the transpose conjugate of $x$.
The previous real vector space is an algebra if provided with the
following product 
\begin{equation}
x*y\coloneqq\mu\cdot xy+\overline{\mu}\cdot yx-\frac{1}{3}\text{Tr}\left(xy\right)\text{Id},
\end{equation}
where $\mu=\nicefrac{1}{6}\left(3+\text{i}\sqrt{3}\right)$ and the
juxtaposition is the ordinary associative product between matrices.
Such algebra is called \emph{Okubo algebra} $\mathcal{O}$ if $i=1$
and \emph{split-Okubo algebra} $\mathcal{O}_{s}$ if $i=2$. In both
cases the algebra is flexible and non-unital and it is also a composition
algebra using the norm 
\begin{equation}
n\left(x\right)\coloneqq\frac{1}{6}\text{Tr}\left(x^{2}\right).
\end{equation}
A canonical basis for both algebras is provided by an idempotent element,
i.e. $e*e=e$, which is 
\begin{equation}
e=\left(\begin{array}{ccc}
2 & 0 & 0\\
0 & -1 & 0\\
0 & 0 & -1
\end{array}\right),\label{eq:idempotentElement}
\end{equation}
 together with other seven elements

\begin{align}
\text{i}_{1}=\left(\begin{array}{ccc}
0 & 1 & 0\\
\gamma1 & 0 & 0\\
0 & 0 & 0
\end{array}\right) & ,\,\,\,\,\text{i}_{2}=\left(\begin{array}{ccc}
0 & -\gamma i & 0\\
i & 0 & 0\\
0 & 0 & 0
\end{array}\right),\nonumber \\
\text{i}_{3}=\left(\begin{array}{ccc}
1 & 0 & 0\\
0 & -1 & 0\\
0 & 0 & 0
\end{array}\right) & ,\,\,\,\,\text{i}_{4}=\left(\begin{array}{ccc}
0 & 0 & 1\\
0 & 0 & 0\\
\gamma1 & 0 & 0
\end{array}\right),\\
\text{i}_{5}=\left(\begin{array}{ccc}
0 & 0 & -\gamma i\\
0 & 0 & 0\\
i & 0 & 0
\end{array}\right), & \,\,\,\,\,\,\text{i}_{6}=\left(\begin{array}{ccc}
0 & 0 & 0\\
0 & 0 & 1\\
0 & 1 & 0
\end{array}\right),\text{i}_{7}=\left(\begin{array}{ccc}
0 & 0 & 0\\
0 & 0 & -i\\
0 & i & 0
\end{array}\right),\label{eq:definizione i ottonioniche}
\end{align}
 where $\gamma=1$ in case of $\mathcal{O}$ and $\gamma=-1$ in the
case of $\mathcal{O}_{s}$. Thus, the Okubo algebra $\mathcal{O}$
and its split version $\mathcal{O}_{s}$ are eight-dimensional real
algebras. A proof that they are not isomorphic is easily obtained
as corollary of the following
\begin{prop}
\label{prop:division}The Okubo Algebra $\mathcal{O}$ is a division
algebra, while the split Okubo Algebra $\mathcal{O}_{s}$ has non-trivial
divisors of zero.
\end{prop}

\begin{proof}
Suppose that $d\neq0$ is a non-trivial left or right divisor of zero,
i.e. $d*x=0$ or $x*d=0$, then from 

\begin{align}
\left(d*x\right)*d= & d*\left(x*d\right)=0=n\left(d\right)x,
\end{align}
we have that $d$ is a divisor of zero if and only if $n\left(d\right)=0$,
i.e. $\text{Tr}\left(d^{2}\right)=0$. But the element $d$ is of
the form 
\begin{equation}
d=\left(\begin{array}{ccc}
\xi_{1} & x_{1}+\text{i}\gamma y_{1} & x_{2}+\text{i}\gamma y_{2}\\
\gamma x_{1}-\text{i}y_{1} & \xi_{2} & x_{3}+\text{i}y_{3}\\
\gamma x_{2}-\text{i}y_{2} & x_{3}-\text{i}y_{3} & -\xi_{1}-\xi_{2}
\end{array}\right),
\end{equation}
where $x_{i},y_{i},\xi_{i}\in\mathbb{R}$ and where $\gamma^{2}=1$
and is $\gamma=1$ in the case of Okubo $\mathcal{O}$ while is $\gamma=-1$
in the split case $\mathcal{O}_{s}$. Therefore, the norm is 
\begin{equation}
n\left(d\right)=\frac{1}{3}\left(\gamma x_{1}^{2}+\gamma x_{2}^{2}+x_{3}^{2}+\gamma y_{1}^{2}+\gamma y_{2}^{2}+y_{3}^{2}+\xi_{1}^{2}+\xi_{2}^{2}+\xi_{1}\xi_{2}\right),\label{eq:norma-d}
\end{equation}
which yields that $\text{Tr}\left(d^{2}\right)=0$ is not possible
in the case of $\gamma=1$ , where as it is easily obtained in the
case $\gamma=-1$.
\end{proof}

\paragraph*{Automorphisms}

The automorphism group of Okubo algebras is a Lie group of $A_{2}$
type \cite{Okubo 1978,ElduQueAut}, more specifically $\text{Aut}\left(\mathcal{O}\right)\cong SU\left(3\right)$
while $\text{Aut}\left(\mathcal{O}_{s}\right)\cong SU\left(2,1\right)$
and therefore $\mathfrak{der}\left(\mathcal{O}\right)\cong\mathfrak{su}\left(3\right)$,
while $\mathfrak{der}\left(\mathcal{O}_{s}\right)\cong\mathfrak{su}\left(2,1\right).$
It is here worth noting that every automorphism also preserves the
norm. Indeed, if $\varphi$ is an automorphism of $\mathcal{O}$ then
from $\varphi\left(x*y\right)=\varphi\left(x\right)*\varphi\left(y\right)$,
we have 
\begin{align}
\varphi\left(\left(x*y\right)*x\right) & =\varphi\left(n\left(x\right)y\right),
\end{align}
but also 
\begin{align}
\left(\varphi\left(x\right)*\varphi\left(y\right)\right)*\varphi\left(x\right)=n\left(\varphi\left(x\right)\right)\varphi\left(y\right),
\end{align}
so that 
\begin{equation}
n\left(x\right)=n\left(\varphi\left(x\right)\right).
\end{equation}
The group of automorphisms is therefore a subgroup of the orthogonal
group $O\left(\mathcal{O}\right)$. Finally, from the analysis of
the norm as a function over an eight-dimensional real vector space,
it is easy to see from (\ref{eq:norma-d}) that $\text{Spin}\left(\mathcal{O}\right)=\text{Spin}\left(8\right)$
and $\text{Spin}\left(\mathcal{O}_{s}\right)=\text{Spin}\left(4,4\right)$.
From a Lie theoretical point of view, while in the case of octonions,
the automorphism $\text{Aut}\left(\mathcal{\mathbb{O}}\right)\simeq\text{G}_{2}$
is not a maximal subgroup in $\text{O}\left(8\right)$, since $\text{G}_{2}\subsetneq\text{O}\left(7\right)\subsetneq\text{O}\left(8\right)$,
in this case we are considering the maximal and non-symmetric embadding
of $\text{A}_{2}$ into $\text{D}_{4}$, such that the $8_{v}$,$8_{s}$,
and $8_{c}$ of $\text{D}_{4}$ all remain irreducible.

\paragraph*{From Michel-Radicati to Okubo }

Okubo algebras enjoy also an interesting interpretation as deformation
of the Michel-Radicati algebra which was introduced in \cite{Michel Radicati 1970}
and whose structure constants were used by Günyadin and Zagermann
to construct unified Maxwell-Einstein supergravity theories (MESGTs)
in $D=5$ space-time dimensions in \cite{G=0000FCnaydin 2003}. The
heuristic motivation for the introduction of the Michel-Radicati product
is that, given two three by three traceless Hermitian matrix over
$\mathbb{C}$, the Jordan product does not yield to a traceless matrices
and therefore traceless Hermitian matrices do not form a Jordan subalgebra
of $\mathfrak{J}_{3}\left(\mathbb{C}\right)$ with the Jordan product.
The problem is avoided defining a new product
\begin{equation}
x\star y\coloneqq\frac{1}{2}xy+\frac{1}{2}yx-\frac{1}{3}\text{tr}\left(xy\right)\text{Id},
\end{equation}
which defines the Michel-Radicati algebra (namely a consistent definition
of the traceless subalgebra of $\mathfrak{J}_{3}\left(\mathbb{C}\right)$
). Thus, the Okubo algebras $\mathcal{O}$ and $\mathcal{O}_{s}$,
might also be interpreted as specific casees of deformed Michel-Radicati
algebras with the following product
\begin{equation}
x\star_{\theta}y\coloneqq\left(\frac{1}{2}+\text{i}\theta\right)xy+\left(\frac{1}{2}-\text{i}\theta\right)yx-\frac{1}{3}\text{Tr}\left(xy\right)\text{Id},\label{eq:deformed product1}
\end{equation}
where the deformation $\theta\in\mathbb{R}$. The $\theta$-deformation
of the Michel-Radicati product enjoys numerous interesting properties,
since the Michel-Radicati algebras is always \emph{flexible}, i.e.
\begin{equation}
x\star_{\theta}\left(y\star_{\theta}x\right)=\left(x\star_{\theta}y\right)\star_{\theta}x,
\end{equation}
and\emph{ Lie-admissible}, i.e. the product $\left[x,y\right]_{\theta}=x\star_{\theta}y-y\star_{\theta}x$
yields to a Lie algebra. The peculiarity of the Okubo algebras $\mathcal{O}$
and $\mathcal{O}_{s}$, is that $\theta=\pm\nicefrac{1}{2\sqrt{3}}$
is the unique value of the Michel-Radicati deformation parameter such
that the deformed product gives rise to a \emph{composition algebra},
i.e. 
\begin{equation}
n\left(x\star_{\theta}y\right)=n\left(x\right)n\left(y\right)\,\,\,\text{iff}\,\,\,\,\theta=\pm\frac{1}{2\sqrt{3}}.
\end{equation}
Moreover, the traceful part of the deformed product (\ref{eq:deformed product1}),
i.e. 
\begin{equation}
x\circ_{\theta}y\coloneqq\left(\frac{1}{2}+\text{i}\theta\right)xy+\left(\frac{1}{2}-\text{i}\theta\right)yx,\label{eq:deformation Jordan}
\end{equation}
gives rise to a non-commutative algebra that is again \emph{Lie-admissible},
i.e. $\left[x,y\right]'_{\theta}=x\circ_{\theta}y-y\circ_{\theta}x$
defines a Lie algebra, but, even more interestingly, the $\circ_{\theta}$
product also satisfies the Jordan identity for every $\theta\in\mathbb{R}$,
thus defining a $\theta$-parametrized family of \emph{non-commutative
Jordan algebras}. Okubo algebras $\mathcal{O}$ and $\mathcal{O}_{s}$
might therefore be interpreted as the only composition algebras resulting
as the traceless sector of a non-commutative Jordan algebra obtained
by deforming the Jordan product over Hermitian matrices as in (\ref{eq:deformation Jordan}). 

Considering three by three matrices, it is amusing to note that the
algebra $\mathfrak{J}_{3}\left(\mathbb{C}\right)_{0}$ of traceless
Hermitian matrices over $\mathbb{C}$ provides the unique case in
which the algebra itself is isomorphic to the Lie algebra of its derivations.
This is actually the case $N=3$ of a general result, corresponding
to the realization of the generators of $\mathfrak{su}\left(n\right)$
as Hermitian traceless $n$ by $n$ matrices over $\mathbb{C}$, i.e.
$\mathfrak{J}_{n}\left(\mathbb{C}\right)_{0}$ (e.g. see \cite{Michel Radicati 1970,Michel 1973}).

\paragraph*{From Okubo to octonions}

From now on we will leaving aside the Okubo split algebra $\mathcal{O}_{s}$
and focus on the real division Okubo algebra $\mathcal{O}$ which
enjoys the important feature of being in a very tight relation with
octonions $\mathbb{O}$. Previously in this section, we mentioned
how Okubo algebra $\mathcal{O}$ can be considered as the Petersson
algebra obtained from the split octonion algebra $\mathbb{O}_{s}$
(cfr. \cite{Elduque Gradings SC}). On the other side, octonions $\mathbb{O}$
can be obtained as a deformation from the Okubo algebra $\mathcal{O}$
through the use of the new product 
\begin{equation}
x\cdot y\coloneqq\left(e*x\right)*\left(y*e\right),
\end{equation}
where $x,y\in\mathcal{O}$ and $e$ is an idempotent of $\mathcal{O}$,
such as the one given by (\ref{eq:idempotentElement}). Nevertheless,
it is important to stress out that the following construction is indipendent
from the explicit form of the idempotent that are, in fact, all conjugate
under the automorphism group \cite[Thm. 20]{ElduQueAut}. Following
\cite{ElduQueAut}, it is easy to show that since $e*e=e$ and $n\left(e\right)=1$,
for every $x\in\mathcal{O}$ the element $e$ acts as a left and right
identity, i.e. 
\begin{align}
x\cdot e & =e*x*e=n\left(e\right)x=x,\\
e\cdot x & =e*x*e=n\left(e\right)x=x.
\end{align}
Moreover, since the Okubo algebra is a composition algebra, the same
norm $n$ enjoys the following relation 
\begin{equation}
n\left(x\cdot y\right)=n\left(\left(e*x\right)*\left(y*e\right)\right)=n\left(x\right)n\left(y\right),
\end{equation}
which means that $\left(\mathcal{O},\cdot,n\right)$ is a unital composition
algebra of real dimension $8$ and, being also a division algebra,
is therefore isomorphic to the algebra of octonions $\mathbb{O}$,
as noted by Okubo himself \cite{Okubo 1978,Okubo 78c}. 

\subsection*{Trivolution}

Even though Okubo algebra $\mathcal{O}$ is not unital, some interesting
features of the unity (albeit not all) are recovered by the use of
an idempotent $e$. While in unital composition algebras the canonical
conjugation, i.e.
\begin{equation}
\overline{x}=\left\langle x,1\right\rangle 1-x,\label{eq:coniugazione}
\end{equation}
has the notable properties of $x\cdot\overline{x}=n\left(x\right)1$
and of $\overline{x\cdot y}=\overline{y}\cdot\overline{x},$ in Okubo
algebras the map $x\longrightarrow\left\langle x,e\right\rangle e-x$
is a map of order two but it is not an automorphism, nor an anti-automorphism.
Similarly, in a para-Hurwitz algebra, which is a non-unital algebra
but rather it is endowed with a paraunit $e$, the maps 
\begin{align}
x & \longrightarrow L_{e}\left(x\right)=e*x,\label{eq:azione a sinistra}\\
x & \longrightarrow R_{e}\left(x\right)=x*e,\label{eq:azioni a destra}
\end{align}
are a conjugation since $e*x=\overline{x}$ and 
\begin{align}
x*L_{e}\left(x\right) & =n\left(x\right)e,\\
R_{e}\left(x\right)*x & =n\left(x\right)e.
\end{align}
On the other hand, in the Okubo algebra $\mathcal{O}$, even though
we still have $R_{e}\circ L_{e}=\text{id}$, both $L_{e}$ and $R_{e}$
are in fact neither are automorphism nor an anti-automorphism. Even
though it is not possible to have a conjugation over Okubo algebra
with the desired properties, we can define something in a similar
fashion such as an order-three automorphism $\tau$, that we will
call here a\emph{ trivolution}, and that can be defined as 
\begin{equation}
x\longrightarrow\tau\left(x\right)\coloneqq\left\langle x,e\right\rangle e-x*e,
\end{equation}
or equivalently as
\begin{align}
x & \longrightarrow\tau\left(x\right)\coloneqq L_{e}^{2}\left(x\right)=e*\left(e*x\right),\label{eq:tau-idempotent}\\
x & \longrightarrow\tau^{2}\left(x\right)\coloneqq R_{e}^{2}\left(x\right)=\left(x*e\right)*e.
\end{align}
It is easy to see that the automorphism $\tau$ is is of order $3$
since, applying flexibility, $R_{e}^{2}\circ L_{e}^{2}=\text{id}$.
It is also worth noting the stunning analogy with the conjugation
expressed for unital composition algebra in (\ref{eq:coniugazione})
and at the same time the analogy with the one expressed for para-Hurwitz
algebras in (\ref{eq:azione a sinistra}). Finally, we have to stress
out that (\ref{eq:tau-idempotent}) shows the deep relation between
the idempotent $e$ and the trivolution $\tau$ : in a certain sense
the choice of the idempotent $e$ and the trivolution $\tau$ are
equivalent (e.g. see \cite[Prop. 18]{ElduQueAut}).
\begin{rem}
It is worth noting that the idempotent $e$ and the order-three automorphism
$\tau$ turn explicit the construction of octonions $\mathbb{O}$
from Okubo algebra $\mathcal{O}$ in an elegant way allowing the definition
of the octonionic product and conjugation, respectively defined by
\begin{align}
x\cdot y & \coloneqq L_{e}\left(x\right)*R_{e}\left(y\right),\\
\overline{x} & \coloneqq L_{e}^{3}\left(x\right)=e*\left(\tau\left(x\right)\right).
\end{align}
\end{rem}

The fixed part of the trivolution is given by 
\begin{equation}
\text{Fix}_{\tau}\left(x\right)\coloneqq\frac{1}{3}\left(x+\tau\left(x\right)+\tau^{2}\left(x\right)\right)=\left(\begin{array}{ccc}
\xi_{1} & 0 & 0\\
0 & \xi_{2} & x_{3}+\text{i}y_{3}\\
0 & x_{3}-\text{i}y_{3} & -\xi_{1}-\xi_{2}
\end{array}\right),\label{eq:fixTau}
\end{equation}
while $\tau$ act on the remaining part $\widetilde{x}=x-\text{Fix}_{\tau}\left(x\right)$
of the element as 
\begin{equation}
\widetilde{x}\longrightarrow\left(\begin{array}{ccc}
0 & -\overline{\alpha}\left(x_{1}+\text{i}y_{1}\right) & -\overline{\alpha}\left(x_{2}+\text{i}y_{2}\right)\\
\alpha\left(x_{1}-\text{i}y_{1}\right) & 0 & 0\\
\alpha\left(x_{2}-\text{i}y_{2}\right) & 0 & 0
\end{array}\right),\label{eq:spinors}
\end{equation}
where $\alpha=\left(1+\text{i}\sqrt{3}\right)/2$ and $x_{1},x_{2},x_{3},y_{1},y_{2},y_{3}\in\mathbb{R}$.
It might be of interest to note that the r.h.s. of (\ref{eq:fixTau})
corresponds to $\left(\mathbb{R}\oplus\mathfrak{J}_{2}\left(\mathbb{C}\right)\right)_{0}\simeq\mathfrak{J}_{2}\left(\mathbb{C}\right)$,
i.e. the simple Jordan algebra of rank-2 over $\mathbb{C}$, given
by two by two Hermitian matrices over $\mathbb{C}$.
\begin{rem}
It is worth remarking that $\mathcal{O}$ is a Lie-admissible algebra,
i.e. $\mathcal{O}$ is a Lie algebra, namely $\mathfrak{su}\left(3\right)$,
if endowed with the commutator $\left[x,y\right]=x*y-y*x$, and that
$\tau$ realizes a $\mathbb{Z}_{2}$-grading of the Lie algebra, i.e.
$\mathcal{O}=\mathfrak{g}_{\overline{0}}\oplus\mathfrak{g}_{\overline{1}},$
where $\mathfrak{g}_{\overline{0}}=\left\{ x\in\mathcal{O};x=\tau\left(x\right)\right\} $
is the Lie subalgebra of elements fixed by $\tau$ and $\mathfrak{g}_{\overline{1}}$
is the complement, i.e. the real vector space generated by elements
of the form (\ref{eq:spinors}). The $\mathbb{Z}_{2}$-grading of
the algebra means that, i.e. $v*v',s\ast s'\in\mathfrak{g}_{\overline{0}}$,
and $s*v,v*s\in\mathfrak{g}_{\overline{1}}$, for every $v,v'\in\mathfrak{g}_{\overline{0}}$
and $s,s'\in\mathfrak{g}_{\overline{1}}$, which obviously extends
to a grading of the Lie algebra such that $\left[v,v'\right],\left[s,s'\right]\in\mathfrak{g}_{\overline{0}}$,
and $\left[s,v\right]\in\mathfrak{g}_{\overline{1}}$.
\end{rem}

\section{The Okubic line}

If in the previous sections we were often dealing with both Okubo
algebra $\mathcal{O}$ and split-Okubo algebra $\mathcal{O}_{s}$,
it is now important to stress out that all the treatment we will develop
from now on, it holds for the Okubo algebra only, and cannot be extended
in a straightforward way to the split-Okubo algebra, which is not
division algebra (cfr. Prop.1). Extensions of the subsequent treatment
to split-Okubo algebra will be considered in future works but momentarily
left aside.

We now construct the Okubic projective line starting from an extended
quadric over $\mathcal{O}$. With this in mind, we start with a real
vector space $\widehat{\mathcal{O}}\cong\mathcal{O}\times\mathbb{R}\times\mathbb{R}$
and a quadratic form $b$ from $\widehat{\mathcal{O}}$ to $\mathbb{R}$
defined as 
\begin{equation}
b\left(x,\xi_{1},\xi_{2}\right)=n\left(x\right)-\xi_{1}\xi_{2},
\end{equation}
which is an extension of $n$ over $\widehat{\mathcal{O}}$. We then
consider the quadric 
\begin{equation}
Q=\left\{ \mathbb{R}v\,\,;\,\,v\in\widehat{\mathcal{O}}\setminus\left\{ 0\right\} ,\,\,\,b\left(v\right)=0\right\} ,
\end{equation}
that obiously lies in $P\mathbb{R}^{9}$ and, since it is closed,
is therefore compact. Thus, to be explicit, points on the Okubic projective
line are elements of the form $\mathbb{R}\left(x,\xi_{1},\xi_{2}\right)$
where $x\in\mathcal{O}$, $\xi_{1},\xi_{2}\in\mathbb{R}$ fulfilling
$n\left(x\right)-\xi_{1}\xi_{2}=0$. Our aim now is to show that this
compact set is a one-point compactification of the Okubic algebra.
This proves the following
\begin{prop}
The map from $\mathcal{O}\cup\left\{ \infty\right\} $ to $Q$ given
by
\begin{align}
x & \mapsto\mathbb{R}\left(x,n\left(x\right),1\right),\\
\infty & \mapsto\mathbb{R}\left(0,1,0\right),
\end{align}
is an homeomorphism, showing that $Q$ is homeomorphic to the one-point
compactification of the Okubo algebra $\mathcal{O}$. 
\end{prop}

\begin{proof}
Since the map is obviously continuous it suffices to see that also
the inverse is continuous. First of all we need to notice that, since
$\mathcal{O}$ is a division algebra, $n\left(x\right)=0$ if and
only if $x=0$ and therefore if $\xi_{2}=0$ then the only $\xi_{1}$
for which a vector $v\in\widehat{\mathcal{O}}\setminus\left\{ 0\right\} $
is such that $b\left(v\right)=0$ is for multiple of $\xi_{1}=1$.
This means that $\mathbb{R}\left(0,1,0\right)$ is the only point
of the quadric that passes through the hyperplane $H=\left\{ \xi_{2}=0\right\} $
in $P\mathbb{R}^{9}$. This means that the image of all $x\in\mathcal{O}$
are in $P\mathbb{R}^{9}\setminus H$, which is homeomorphic to $\mathbb{R}^{9}$.
Moreover if $\mathbb{R}\left(x,\xi_{1},1\right)$ is in the quadric
$Q$ then $n\left(x\right)-\xi_{1}=0$, i.e. $\xi_{1}=n\left(x\right)$,
and therefore the map from $\mathbb{R}\left(x,n\left(x\right),1\right)\mapsto x$
is a continuous map.
\end{proof}

\section{The affine and projective plane over the Okubo algebra}
\begin{center}
\begin{figure}
\centering{}\includegraphics[scale=0.3]{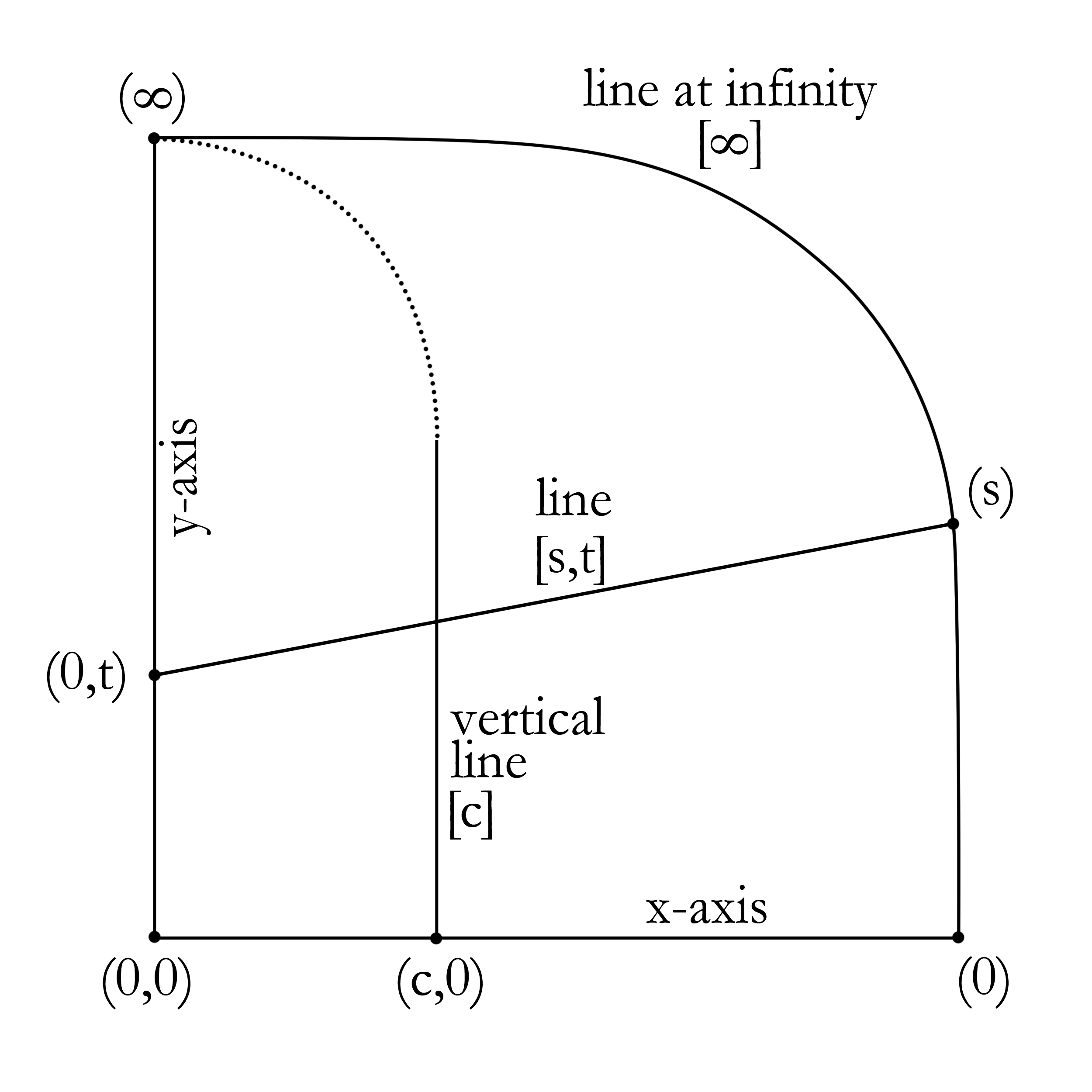}\caption{\label{fig:The affine plane}Representation of the affine plane: $\left(0,0\right)$
represents the origin, $\left(0\right)$ the point at the infinity
on the $x$-axis, $\left(s\right)$ is the point at infinity of the
line $\left[s,t\right]$ of slope $s$ while $\left(\infty\right)$
is the point at the infinity on the $y$-axis and of vertical lines
$\left[c\right]$.}
\end{figure}
\par\end{center}

Here we define the affine and the projective plane following section
3 and 4 of our previous work \cite{Corr-OkuboSpin}. The Okubic projective
plane can be equivalently defined as completion of the affine Okubic
plane or directly through the use of a modification of the Veronese
coordinates. We here present both constructions and we also determine
the maps that realizes the isomorphisms between the two. A straightforward
proof of the isomorphism between the two planes and of the preservation
through the map of incidence properties of points and lines can be
found in \cite[Sec. 4]{Corr-OkuboSpin}, here we will just present
the results in order to set everything for the next two sections.

\subsection*{The Okubic affine plane. }

The definition of the Okubic affine plane $\mathscr{A}_{2}\left(\mathcal{O}\right)$
is easily achieved by defining a \emph{point} on the affine plane
as the two coordinates $\left(x,y\right)$ with $x,y\in\mathcal{O}$,
and a \emph{line} of the affine plane as $\left[s,t\right]$ with
$s,t\in\mathcal{O}$ where we intended $\left[s,t\right]$ as the
set of Okubic elements $\left\{ \left(x,s*x+t\right):x\in\mathcal{O}\right\} .$
Lines on the Okubic plane are therefore determined by two Okubic elements
$s$ and $t$ that are called the \emph{slope} and the \emph{offset}
of the line, respectively. Vertical lines are identified by $\left[c\right]$
that denotes the set $\left\{ c\right\} \times\mathcal{O}$, where
$c\in\mathcal{O}$ represents the intersection of the line with the
$x$ axis. It is easily seen that since $\mathcal{O}$ is a division
algebra the Okubic affine plane satisfies all incidence axioms of
the affine geometry \cite[Sec. 3]{Corr-OkuboSpin}. 

The completion of the Okubic affine plane $\overline{\mathscr{A}_{2}}\left(\mathcal{O}\right)$
is obtained adding a line at infinity $\left[\infty\right]$, i.e.
\[
\left[\infty\right]=\left\{ \left(s\right):s\in\mathcal{O}\cup\left\{ \infty\right\} \right\} ,
\]
where $\left(s\right)$ identifies the end point at infinity of a
line with slope $s\in\mathcal{O}\cup\left\{ \infty\right\} $. Finally,
we define $\left(\infty\right)$ as the point at infinity of $\left[\infty\right]$.
It is easy to verify that this construction preserves the property
of a unique line joining two different points and that two lines intersect
at infinity if and only if they are parallel, i.e. have the same slope
$\left(s\right)$. Resuming the whole notation, as in Fig.\ref{fig:The affine plane},
we have three set of coordinates that indentify all the points in
the completion of the affine plane, i.e. $\left(x,y\right),\left(s\right)$
and $\left(\infty\right)$. 

\subsection*{The Okubic projective plane.}

Resuming our construction in \cite{Corr-OkuboSpin}, inspired by \cite{Compact Projective},
let $V\cong\mathbb{\mathcal{O}}^{3}\times\mathbb{R}^{3}$ be a real
vector space, with elements of the form 
\[
\left(x_{\nu};\lambda_{\nu}\right)_{\nu}=\left(x_{0},x_{1},x_{2};\lambda_{0},\lambda_{1},\lambda_{2}\right)
\]
where $x_{\nu}\in\mathcal{O}$, $\lambda_{\nu}\in\mathbb{R}$ and
$\nu=0,1,2$. The vector $w\in V$ is called \emph{Okubo-Veronese}
if

\begin{align}
\lambda_{0}x_{0} & =x_{1}*x_{2},\,\,\lambda_{1}x_{1}=x_{2}*x_{0},\,\,\lambda_{2}x_{2}=x_{0}*x_{1}\label{eq:Ver-1-1}\\
n\left(x_{0}\right) & =\lambda_{1}\lambda_{2},\,n\left(x_{1}\right)=\lambda_{2}\lambda_{0},n\left(x_{2}\right)=\lambda_{0}\lambda_{1}.\label{eq:Ver-2-1}
\end{align}
Now we consider the subspace $H\subset V$ be of Okubo-Veronese vectors.
It is straightforward to see that if $w=\left(x_{\nu};\lambda_{\nu}\right)_{\nu}$
is an Okubo-Veronese vector then also $\mu w=\mu\left(x_{\nu};\lambda_{\nu}\right)_{\nu}$
is such a vector for every $\mu\in\mathbb{R}$, and, therefore, $\mathbb{R}w\subset H$.
We define the \emph{Okubic projective plane} $\mathbb{P}^{2}\mathcal{O}$
as the geometry having this 1-dimensional subspaces $\mathbb{R}w$
as points, i.e. 
\begin{equation}
\mathbb{P}^{2}\mathcal{O}\coloneqq\left\{ \mathbb{R}w:w\in H\setminus\left\{ 0\right\} \right\} .\label{eq:punti piano pro}
\end{equation}

On the projective plane $\mathbb{P}^{2}\mathcal{O}$ lines are defined
through the use of a symmetric bilinear form $\beta$ that is the
extension of the polar form (\ref{eq:linearization}) of the quadratic
norm $n$ over the Okubo algebra $\mathcal{O}$, i.e.
\begin{equation}
\beta\left(v,w\right)\coloneqq\stackrel[\nu=0]{2}{\sum}\left(\left\langle x_{\nu},y_{\nu}\right\rangle +\lambda_{\nu}\eta_{\nu}\right),
\end{equation}
where $v=\left(x_{\nu};\lambda_{\nu}\right)_{\nu}$ and $w=\left(y_{\nu};\eta_{\nu}\right)_{\nu}$
are Okubo-Veronese vectors in $H\subset V\cong\mathbb{\mathcal{O}}^{3}\times\mathbb{R}^{3}$.
The \emph{lines} $\ell_{w}$ in the projective plane $\mathbb{P}^{2}\mathcal{O}$
are the orthogonal spaces of a vector $w\in H$, i.e. 
\begin{equation}
\ell_{w}\coloneqq w^{\bot}=\left\{ z\in V:\beta\left(z,w\right)=0\right\} ,
\end{equation}
and, clearly, a point $\mathbb{R}v$ is incident to the line $\ell_{w}$
when $\mathbb{R}v\subseteq$$w^{\bot}$. 
\begin{rem}
It is worth expliciting how the norm $n$ defined over the symmetric
composition algebra $\mathbb{\mathcal{O}}$ is intertwined with the
geometry of the plane. This relations is manifest when we consider
the quadratic form of the bilinear symmetric form $\beta$, i.e.

\begin{equation}
\left\Vert v\right\Vert \coloneqq\beta\left(v,v\right)=2n\left(x_{0}\right)+2n\left(x_{1}\right)+2n\left(x_{2}\right)+\lambda_{0}^{2}+\lambda_{1}^{2}+\lambda_{2}^{2},\label{eq: norm of an element}
\end{equation}
\end{rem}

\subsection*{Correspondence between projective and affine plane}

The identification of the affine Okubic plane with the projective
one can be explicited defining the map that sends a point of the affine
plane to the projective point in $V\cong\mathbb{\mathcal{O}}^{3}\times\mathbb{R}^{3}$,
i.e. 
\begin{align}
\left(x,y\right) & \mapsto\mathbb{R}\left(x,y,x*y;n\left(y\right),n\left(x\right),1\right),\label{eq:corrispondenza1}\\
\left(x\right) & \mapsto\mathbb{R}\left(0,0,x;n\left(x\right),1,0\right),\\
\left(\infty\right) & \mapsto\mathbb{R}\left(0,0,0;1,0,0\right).
\end{align}
Since the Okubo algebra is a symmetric composition algebra, then the
map is well defined. Indeed, from (\ref{eq:Ver-2-1}) we note that
\begin{equation}
n\left(x\right)=\lambda_{1},\,\,\,n\left(y\right)=\lambda_{0},
\end{equation}
and since Okubo is a composition algebra then 
\begin{equation}
n\left(x*y\right)=n\left(x\right)n\left(y\right).
\end{equation}
Since Okubo algebra is flexible we also have that (\ref{eq:Ver-1-1})
are satisfied and 
\begin{align*}
\lambda_{0}x & =y*\left(x*y\right)=n\left(y\right)x,\,\,\\
\lambda_{1}y & =\left(x*y\right)*x=n\left(x\right)y,\\
\lambda_{2}\left(x*y\right) & =x*y,
\end{align*}
and therefore $\mathbb{R}\left(x,y,x*y;n\left(y\right),n\left(x\right),1\right)$
is a point in the Okubic projective plane. As for the converse, if
a point $p$ of coordinates $\left(x_{\nu};\lambda_{\nu}\right)_{\nu}$
is in $\mathbb{P}^{2}\mathcal{O}$ then it satisfies (\ref{eq:Ver-2-1})
and it has one of the $\lambda_{\nu}$ different from zero. 

In \cite[Sec. 4]{Corr-OkuboSpin} we have shown that the map (\ref{eq:corrispondenza1})
that the correspondence is a bijection that sends affine points incident
to an affine line into projective points incident to the same projective
line. 

\section{A deformation of the Okubic Albert algebra }

Since there is a well known geometric correspondence \cite{corr Notes Octo}
between rank-one idempotents of a simple rank-three Jordan algebra
$\mathfrak{J}_{3}\left(\mathbb{K}\right)$ over a Hurwitz algebras
$\mathbb{K}\in\left\{ \mathbb{R},\mathbb{C},\mathbb{H},\mathbb{O}\right\} $
and elements of the projective planes $\mathbb{K}P^{2}$ that was
introduced by Jordan, Von Neumann and Wigner \cite{Jordan} in 1934,
it thus appears logical to investigate the Jordan algebras of rank
three over the Okubo algebra that can be potentially linked to the
Okubic projective plane introduced above. Great work in this sense
was done by Elduque who first introduced an equivalent of Tits-Freudenthal
Magic Square for flexible composition algebras in \cite{EldMS1,EldMS2}
and later on defined an \emph{Okubic Albert algebra} \cite{Elduque Gradings SC}
which he found to be isomorphic to the exceptional Jordan algebra
over split-octonions $\mathfrak{J}_{3}\left(\mathbb{O}_{s}\right)$
and could be therefore used to define a $\mathbb{Z}_{3}$-grading
over the exceptional Lie algebra $\mathfrak{f}_{4}$. 

We now introduce a slight variation of the \emph{Okubic Albert algebra}
introduced by Elduque in \cite[Theo 5.15]{ElDuque Comp}. Let be $\mathbb{A}_{q}\left(\mathcal{O}\right)$
the vector space $\mathcal{O}^{3}\oplus\mathbb{R}^{3}$ equipped with
the following commutative product

\begin{equation}
\left(x_{\nu};\lambda_{\nu}\right)\circ\left(y_{\eta};\mu_{\eta}\right)\coloneqq\left(\begin{array}{c}
\frac{\lambda_{1}+\lambda_{2}}{2}y_{0}+\frac{\mu_{1}+\mu_{2}}{2}x_{0}+q\left(x_{1}*y_{2}+y_{1}*x_{2}\right)\\
\frac{\lambda_{0}+\lambda_{2}}{2}y_{1}+\frac{\mu_{0}+\mu_{2}}{2}x_{1}+q\left(x_{2}*y_{0}+y_{2}*x_{0}\right)\\
\frac{\lambda_{0}+\lambda_{1}}{2}y_{2}+\frac{\mu_{0}+\mu_{1}}{2}x_{2}+q\left(x_{0}*y_{1}+y_{0}*x_{1}\right)\\
\lambda_{0}\mu_{0}+\left\langle x_{1},y_{1}\right\rangle +\left\langle x_{2},y_{2}\right\rangle \\
\lambda_{1}\mu_{1}+\left\langle x_{0},y_{0}\right\rangle +\left\langle x_{2},y_{2}\right\rangle \\
\lambda_{2}\mu_{2}+\left\langle x_{0},y_{0}\right\rangle +\left\langle x_{1},y_{1}\right\rangle 
\end{array}\right),\label{eq:JordanElduque}
\end{equation}
where $\nu,\eta=1,2,3$ and $\left\langle x,x\right\rangle =n\left(x\right)$.
Setting $q=1$ we recover the \emph{Okubic Albert Algebra} $\mathbb{A}\left(\mathcal{O}\right)$
introduced by Elduque which is a Jordan Algebra that it was shown
to be isomorphic to $\mathfrak{J}_{3}\left(\mathbb{O}_{s}\right)$.
We also notice, for the first time to our knowledge, that also $q=-1$
gives a Jordan Algebra, that we would expect to be isomorphic to the
exceptional Jordan algebra $\mathfrak{J}_{3}\left(\mathbb{O}\right)$,
even though we leave a full investigation on the matter for a future
article. For the purposes of this article, we will set $q=\nicefrac{1}{2}$,
which gives a new algebra $\mathbb{A}_{1/2}\left(\mathcal{O}\right)$
which is unital, with unit $\left(0,0,0;1,1,1\right)$, commutative
and flexible but is neither alternative nor Jordan.

It is worth noting that using the correspondence
\begin{equation}
\left(x_{\nu};\lambda_{\nu}\right)\longrightarrow\left(\begin{array}{ccc}
\lambda_{0} & x_{2} & \overline{x}_{1}\\
\overline{x}_{2} & \lambda_{1} & x_{0}\\
x_{1} & \overline{x}_{0} & \lambda_{2}
\end{array}\right),
\end{equation}
the product in (\ref{eq:JordanElduque}) give rise to the usual Jordan
algebras $\mathfrak{J}_{3}\left(\mathbb{\mathbb{K}}\right)$ when
$\mathbb{K}$ is an Hurwitz algebra (see \cite[sec. 5.3]{Elduque Gradings SC}).
It is therefore natural to define a linear, a quadratic and a cubic
norm on the deformed Okubic Albert algebra $\mathbb{A}_{1/2}\left(\mathcal{O}\right)$
, i.e. the \emph{trace} of an element as the linear norm 
\begin{equation}
\text{Tr}\left(\left(x_{\nu};\lambda_{\nu}\right)\right)\coloneqq\lambda_{0}+\lambda_{1}+\lambda_{2},
\end{equation}
the quadratic norm or simply the \emph{norm} of an element, i.e. 
\begin{equation}
\left\Vert \left(x_{\nu};\lambda_{\nu}\right)\right\Vert \coloneqq2n\left(x_{0}\right)+2n\left(x_{1}\right)+2n\left(x_{2}\right)+\lambda_{0}^{2}+\lambda_{1}^{2}+\lambda_{2}^{2},
\end{equation}
which in addition to being an extension to the Okubic algebra of the
definition in the Hurwitz case, is also consistent with (\ref{eq: norm of an element}).
From the norm we obtain its polarization, i.e. the inner product given
by 
\begin{equation}
\left\langle \left(x_{\nu};\lambda_{\nu}\right),\left(y_{\nu};\mu_{\nu}\right)\right\rangle _{\mathbb{A}}\coloneqq\left\Vert \left(x_{\nu}+y_{\nu};\lambda_{\nu}+\mu_{\nu}\right)\right\Vert -\left\Vert \left(x_{\nu};\lambda_{\nu}\right)\right\Vert -\left\Vert \left(y_{\nu};\mu_{\nu}\right)\right\Vert .
\end{equation}
Finally, we introduce the \emph{cubic norm} $N$ of an element, i.e.
\begin{align}
N\left(\left(x_{\nu};\lambda_{\nu}\right)\right) & \coloneqq\lambda_{0}\lambda_{1}\lambda_{2}-\left(\lambda_{0}n\left(x_{0}\right)+\lambda_{1}n\left(x_{1}\right)+\lambda_{2}n\left(x_{2}\right)\right)+2\left\langle \left(x_{0}*e\right)*\left(x_{1}*x_{2}\right),e\right\rangle 
\end{align}
where $n\left(x\right)$ is the Okubic norm of $x$ and $e$ is the
idempotent element defined in (\ref{eq:idempotentElement}). 

In analogy with the $\mathfrak{J}_{3}\left(\mathbb{\mathbb{K}}\right)$
case, we will define an idempotent element of $\mathbb{A}_{1/2}\left(\mathcal{O}\right)$
to be \emph{rank-1}, an idempotent element that has cubic norm $N\left(\left(x_{\nu};\lambda_{\nu}\right)\right)=0$
and $\text{Tr}\left(\left(x_{\nu};\lambda_{\nu}\right)\right)=1$.
Indeed, it is well known that on  idempotent elements $e$ of $\mathfrak{J}_{3}\left(\mathbb{\mathbb{K}}\right)$
we have that $\text{Tr}\left(e\right)=\text{rank}\left(e\right)$
thus rank-1 idempotents are those and only those with trace equal
to one (e.g. see \cite[Note 16.8]{Compact Projective}).

\subsection*{Correspondence with the Okubic projective plane}

Even though $\mathbb{A}_{1/2}\left(\mathcal{O}\right)$ is not a Jordan
algebra, it yet retains the geometrical analogy with $\mathfrak{J}_{3}\left(\mathbb{K}\right)$
when $\mathbb{\mathbb{K}}\in\left\{ \mathbb{R},\mathbb{C},\mathbb{H},\mathbb{O}\right\} $
since the condition of idempotency on rank-1 elements, i.e. $\left(x_{\nu};\lambda_{\nu}\right)\circ\left(x_{\nu};\lambda_{\nu}\right)=\left(x_{\nu};\lambda_{\nu}\right)$,
$N\left(\left(x_{\nu};\lambda_{\nu}\right)\right)=0$ and $Tr\left(\left(x_{\nu};\lambda_{\nu}\right)\right)=1$,
yields to
\begin{align}
\left(\begin{array}{c}
\left(1-\lambda_{0}\right)x_{0}+x_{1}*x_{2}\\
\left(1-\lambda_{1}\right)x_{1}+x_{2}*x_{0}\\
\left(1-\lambda_{2}\right)x_{2}+x_{0}*x_{1}\\
\lambda_{0}\lambda_{0}+n\left(x_{1}\right)+n\left(x_{2}\right)\\
\lambda_{1}\lambda_{1}+n\left(x_{0}\right)+n\left(x_{2}\right)\\
\lambda_{2}\lambda_{2}+n\left(x_{0}\right)+n\left(x_{1}\right)
\end{array}\right)= & \left(\begin{array}{c}
x_{0}\\
x_{1}\\
x_{2}\\
\lambda_{0}\\
\lambda_{1}\\
\lambda_{2}
\end{array}\right),\label{eq:idemp condizione}
\end{align}
which is satisfied \emph{iff} (\ref{eq:Ver-2-1}) and (\ref{eq:Ver-1-1})
hold true along with the condition $\lambda_{0}+\lambda_{1}+\lambda_{2}=1$,
which isolate a specific representative of the class $\mathbb{R}\left(x_{\nu};\lambda_{\nu}\right)$.
We therefore have a one-to-one correspondence between points of the
Okubic projective plane as defined in (\ref{eq:punti piano pro})
and rank-1 idempotent elements of the algebra $\mathbb{A}_{1/2}\left(\mathcal{O}\right)$
.

It is worth noting that conditions to (\ref{eq:Ver-2-1}) and (\ref{eq:Ver-1-1})
together with $\text{Tr}\left(\left(x_{\nu};\lambda_{\nu}\right)\right)=1$,
imply that the quadratic norm is identically $1$ since 
\begin{align}
\left\Vert \left(x_{\nu};\lambda_{\nu}\right)\right\Vert  & =2n\left(x_{0}\right)+2n\left(x_{1}\right)+2n\left(x_{2}\right)+\lambda_{0}^{2}+\lambda_{1}^{2}+\lambda_{2}^{2}=\nonumber \\
 & =2\lambda_{1}\lambda_{2}+2\lambda_{2}\lambda_{0}+2\lambda_{0}\lambda_{1}+\lambda_{0}^{2}+\lambda_{1}^{2}+\lambda_{2}^{2}=\nonumber \\
 & =\left(\lambda_{0}+\lambda_{1}+\lambda_{2}\right)^{2}=1,
\end{align}
and, moreover that the cubic norm is identically zero since
\begin{align}
N\left(\left(x_{\nu};\lambda_{\nu}\right)\right) & =\lambda_{0}\lambda_{1}\lambda_{2}-\left(\lambda_{0}n\left(x_{0}\right)+\lambda_{1}n\left(x_{1}\right)+\lambda_{2}n\left(x_{2}\right)\right)+2\left\langle \left(x_{0}*e\right)*\left(x_{1}*x_{2}\right),e\right\rangle =\nonumber \\
 & =-2\lambda_{0}\lambda_{1}\lambda_{2}-2\lambda_{0}\left\langle \left(x_{0}*e\right)*x_{0},e\right\rangle =\nonumber \\
 & =-2\lambda_{0}\lambda_{1}\lambda_{2}-2\lambda_{0}\lambda_{1}\lambda_{2}\left\langle e,e\right\rangle =\nonumber \\
 & =0
\end{align}

\begin{rem}
In the case of the Jordan algebra $\mathfrak{J}_{3}\left(\mathbb{\mathbb{K}}\right)$
over a Hurwitz algebras $\mathbb{\mathbb{K}}\in\left\{ \mathbb{R},\mathbb{C},\mathbb{H},\mathbb{O}\right\} $,
the Veronese conditions constrain the determinant to be zero exactly
as for $\mathbb{A}_{1/2}\left(\mathcal{O}\right)$. Moreover, in case
of the Jordan algebra $\mathfrak{J}_{3}\left(\mathbb{\mathbb{K}}\right)$
the condition of an idempotent to be trace one, i.e. $\lambda_{0}+\lambda_{1}+\lambda_{2}=1$,
is equivalent to the requirement of being a rank-1 idempotent \cite{Compact Projective,corr Notes Octo}.
Therefore, it exists a perfect analogy between the correspondence
of the deformed Okubic Albert algebra $\mathbb{A}_{1/2}\left(\mathcal{O}\right)$
and the Okubic projective plane $\mathcal{O}P^{2}$ and the correspondence
between the usual projective planes $\mathbb{\mathbb{K}}P^{2}$ and
the simple rank-three Jordan algebras $\mathfrak{J}_{3}\left(\mathbb{\mathbb{K}}\right)$. 
\end{rem}

\medskip{}

\begin{rem}
We also have to note that, while $\mathbb{A}_{1/2}\left(\mathcal{O}\right)$
is not a Jordan algebra, the restriction on the idempotent elements
does satisfy the Jordan identity. This is a consequence of the flexibility
of the Okubo algebra. Indeed, if an algebra is flexible, i.e. 
\begin{equation}
\left(x\circ y\right)\circ x=x\circ\left(y\circ x\right),
\end{equation}
then, on the idempotent elements, we have 
\begin{align}
\left(x\circ y\right)\circ\left(x\circ x\right) & =\left(x\circ y\right)\circ x\\
 & =x\circ\left(y\circ x\right)\nonumber \\
 & =x\circ\left(y\circ\left(x\circ x\right)\right),\nonumber 
\end{align}
and thus idempotent elements fulfil the Jordan identity.
\end{rem}

\subsection*{Automorphisms of $\mathbb{A}_{1/2}\left(\mathcal{O}\right)$}

Freudenthal \cite{Freud 1965} and Rosenfeld \cite{Rosenf98,Rosenfeld-1993},
noted how the automorphisms of $\mathfrak{J}_{3}\left(\mathbb{\mathbb{O}}\right)$
could be interpreted as isometries of the octonionic projective plane
$\mathbb{O}P^{2}$, giving rise to an algebro-geometrical realization
of the real compact form of the exceptional Lie group $F_{4\left(-52\right)}$,
while the other real forms, i.e. $F_{4\left(4\right)}$ and $F_{4\left(-20\right)}$,
are obtained as isometry groups of the \emph{split-octonionic projective
plane} $\mathbb{O}_{s}P^{2}$ and the \emph{octonionic hyperbolic
plane} $\mathbb{O}H^{2}$ respectively (see \cite{Corr RealF,corr Notes Octo}
for a recent account). Rosenfeld himself in \cite{Rosenf98,Rosenfeld-1993}
proceeded in relating real forms of exceptional Lie groups with projective
and hyperbolic planes over tensorial products of Hurwitz algebras.
This approach had yielded to a definition of the so-called octonionic
Rosenfeld planes as homogeneous (and symmetric) spaces (see \cite{A Magic Approach}
for a comprehensive treatment). It is therefore interesting to study
the automorphisms of $\mathbb{A}_{1/2}\left(\mathcal{O}\right)$ in
order to extend Freudenthal and Rosenfeld's description to the Okubic
projective plane.

We start and notice that if $\varphi$ is an automorphism of the Okubo
algebra, namely if it belongs to $\text{Aut}\left(\mathcal{O}\right)$,
i.e. $\varphi\left(x*y\right)=\varphi\left(x\right)*\varphi\left(y\right)$,
then it exists an automorphism 
\begin{equation}
\Phi\left(x_{\nu};\lambda_{\nu}\right)=\left(\varphi\left(x_{0}\right),\varphi\left(x_{1}\right),\varphi\left(x_{2}\right);\lambda_{0},\lambda_{1},\lambda_{2}\right)\in\text{Aut}\left(\mathbb{A}_{1/2}\left(\mathcal{O}\right)\right),
\end{equation}
since, as shown in section 2, an automorphism of $\mathcal{O}$ is
also an isometry, i.e. $\left\langle \varphi\left(x\right),\varphi\left(y\right)\right\rangle =\left\langle x,y\right\rangle $.
It is therefore straighforward to realize that $SU\left(3\right)$
is a subgroup of $\text{Aut}\left(\mathbb{A}_{1/2}\left(\mathcal{O}\right)\right)$.
Moreover, we also note that $\mathbb{Z}/3\mathbb{Z}<\text{Aut}\left(\mathbb{A}_{1/2}\left(\mathcal{O}\right)\right)$
since 
\begin{equation}
\tau\left(\left(x_{0},x_{1},x_{2};\lambda_{0},\lambda_{1},\lambda_{2}\right)\right)=\left(x_{2},x_{0},x_{1};\lambda_{2},\lambda_{0},\lambda_{1}\right)
\end{equation}
is an automorphism. On the other hand, a transposition of just two
coordinates does not yield to an automorphism; this is easily seen
considering the element
\begin{equation}
\left(x,0,0;0,0,0\right)\circ\left(0,y,0;0,0,0\right)=\left(0,0,q\left(x*y\right);0,0,0\right).
\end{equation}
Then, suppose $\sigma$ to be an application that switches the coordinates
0 with 1, we then have that 
\begin{align}
\sigma\left(\left(0,0,q\left(x*y\right);0,0,0\right)\right) & =\left(0,0,q\left(x*y\right);0,0,0\right)=\\
 & =\left(0,x,0;0,0,0\right)\circ\left(y,0,0;0,0,0\right),
\end{align}
which happens \emph{iff }$\left(y*x\right)=\left(x*y\right)$, which
would require the product $*$ to be commutative. 

In order to conclude the present investigation, we have to consider
the vector space decomposition of $\mathbb{A}_{1/2}\left(\mathcal{O}\right)$
as 
\begin{equation}
\mathbb{A}_{1/2}\left(\mathcal{O}\right)\simeq\mathcal{O}^{\left(0\right)}\oplus\mathcal{O}^{\left(1\right)}\oplus\mathcal{O}^{\left(2\right)}\oplus\mathbb{R}^{\left(0\right)}\oplus\mathbb{R}^{\left(1\right)}\oplus\mathbb{R}^{\left(2\right)},
\end{equation}
and $w_{i}$ and $\omega_{i}$ are the maps from $\mathcal{O}$ and
$\mathbb{R}$ respectively in $\mathbb{A}_{1/2}\left(\mathcal{O}\right)$,
such that they send the element of the algebra into the $i$ coordinate,
e.g. $w_{2}\left(x\right)\coloneqq\left(0,0,x;0,0,0\right)$ and $\omega_{2}\left(\lambda\right)\coloneqq\left(0,0,0;0,0,\lambda\right)$.
The algebra $\mathbb{A}_{1/2}\left(\mathcal{O}\right)$ has unit $\left(0,0,0;1,1,1\right)$
which can be obtained as the sum of three primitive, i.e. that are
not sum of other idempotents, and orthogonal, i.e. such that $e_{i}\circ e_{j}=0$,
idempotents given by
\begin{align}
e_{0} & \coloneqq\omega_{0}\left(1\right)=\left(0,0,0;1,0,0\right),\\
e_{1} & \coloneqq\omega_{1}\left(1\right)=\left(0,0,0;0,1,0\right),\\
e_{2} & \coloneqq\omega_{2}\left(1\right)=\left(0,0,0;0,0,1\right).
\end{align}
 We thus have that $\left\{ e_{0},e_{1},e_{2}\right\} $ is a complete
system of primitive idempotents. In fact is the only complete system
of primitive idempotents with real coordinates. We now state the following 
\begin{lem}
\label{lem:If--is}If $\Phi$ is an automorphism, then 
\begin{equation}
\Phi\left(e_{i}\right)=e_{\sigma\left(i\right)},
\end{equation}
for a suitable permutation $\sigma\in\mathfrak{S}_{3}$. In other
words, a complete system of primitive idempotents with only real coordinates
must be sent in another complete system of primitive idempotents with
only real coordinates. 
\end{lem}

\begin{proof}
That a complete primitive system of idempotents must go into a complete
primitive system of idempotents it is self-evident. What is left to
show is that a primitive real idempotent can only go in a primitive
real idempotent, where we have called \emph{real idempotents} those
of the form $e=\left(0,0,0;\alpha_{0},\alpha_{1},\text{\ensuremath{\alpha_{2}}}\right)$.
First of all we notice that a primitive real idempotent has only one
non-null real coordinate, since for (\ref{eq:idemp condizione}) all
real coordinates must be idempotent and thus an idempotent of the
form $e=\left(0,0,0;\alpha_{0},\alpha_{1},0\right)$, would be decomposable
into the sum of two idempotents $\left(0,0,0;\alpha_{0},0,0\right)$
and $\left(0,0,0;0,\alpha_{1},0\right)$, thus violating the condition
of being primitive. Therefore, without any loss of generality, let
us consider the left action of a primitive real idempotent observing
the left action of $e_{0}=\left(0,0,0;1,0,0\right)$. Since 
\begin{equation}
L_{e_{0}}\left(y_{\eta};\mu_{\eta}\right)=e_{0}\circ\left(y_{\eta};\mu_{\eta}\right)=\left(0,\frac{1}{2}y_{1},\frac{1}{2}y_{2};\mu_{0},0,0\right),\label{eq:leftaction e0}
\end{equation}
we then have that the image $L_{e_{0}}$ is given by $\mathcal{O}^{\left(1\right)}\oplus\mathcal{O}^{\left(2\right)}\oplus\mathbb{R}^{\left(0\right)}$.
Now let us consider the kernel of the left action of $e_{0}$, i.e.
\begin{equation}
\text{ker}L_{e_{0}}=\left\{ \left(y_{\eta};\mu_{\eta}\right)\in\mathbb{A}_{1/2}\left(\mathcal{O}\right):e_{0}\circ\left(y_{\eta};\mu_{\eta}\right)=0\right\} .
\end{equation}
We notice that for every $\left(y_{\eta};\mu_{\eta}\right)\in\text{ker}L_{e_{0}}$
we must have a corresponding element in $\text{ker}L_{\Phi\left(e_{0}\right)}$
since
\begin{equation}
0=\Phi\left(0\right)=\Phi\left(e_{0}\circ\left(y_{\eta};\mu_{\eta}\right)\right)=\Phi\left(e_{0}\right)\circ\Phi\left(y_{\eta};\mu_{\eta}\right),
\end{equation}
so that $\Phi\left(y_{\eta};\mu_{\eta}\right)\in\text{ker}L_{\Phi\left(e_{0}\right)}$.
We then have 
\begin{equation}
\text{dim}\left(\text{ker}L_{\Phi\left(e_{0}\right)}\right)\geq\text{dim}\left(\text{ker}L_{e_{0}}\right),
\end{equation}
and from (\ref{eq:leftaction e0}) it is clear that $\text{dim}\left(\text{ker}L_{e_{0}}\right)=10$.
If $\Phi\left(e_{0}\right)$ is a primitive real idempotent the condition
is clearly satisfied. Now, let us suppose that $\Phi\left(e_{0}\right)=\varepsilon=\left(x_{\nu};\lambda_{\nu}\right)$
is a primitive, but not of real type, idempotent. We therefore have
that at least one Okubic coordinate $x_{\nu}\in\mathcal{O}$ has to
be different from zero (we also notice from (\ref{eq:JordanElduque}),
that if we have more non-null Okubonic coordinate, then the dimension
of the kernel is smaller, so this is not a restrictive hypothesis).
Let us suppose without any loss of generality that the idempotent
$\varepsilon$ has $x_{0}\neq0$. Since Okubo algebra is a division
algebra, this means that $n\left(x_{0}\right)\neq0$, and since every
idempotent of non-real type satisfies Veronese conditions (\ref{eq:Ver-1-1})
and (\ref{eq:Ver-2-1}) at least $\lambda_{1},\lambda_{2}$ are non-null,
because $n\left(x_{0}\right)=\lambda_{1}\lambda_{2}.$ But then, analyzing
the kernel, of the left action $L_{\varepsilon}$, we have that
\begin{equation}
L_{\varepsilon}\circ\left(y_{\eta};\mu_{\eta}\right)=\left(\begin{array}{c}
\frac{\lambda_{1}+\lambda_{2}}{2}y_{0}+\frac{\mu_{1}+\mu_{2}}{2}x_{0}\\
\frac{\lambda_{2}}{2}y_{1}+q\left(y_{2}*x_{0}\right)\\
\frac{\lambda_{1}}{2}y_{2}+q\left(x_{0}*y_{1}\right)\\
0\\
\lambda_{1}\mu_{1}+\left\langle x_{0},y_{0}\right\rangle \\
\lambda_{2}\mu_{2}+\left\langle x_{0},y_{0}\right\rangle 
\end{array}\right),
\end{equation}
from which one obtain $\text{dim}\left(\text{ker}L_{\varepsilon}\right)=1<\text{dim}\left(\text{ker}L_{e_{0}}\right)$,
thus ruling out the possibility of $\Phi\left(e_{0}\right)$ being
$\varepsilon$.
\end{proof}
Using the Lemma \ref{lem:If--is} we can now investigate all automorphisms.
Let $\Phi\in\text{Aut}\left(\mathbb{A}_{1/2}\left(\mathcal{O}\right)\right)$,
we then have that 
\begin{equation}
\Phi\left(e_{i}\right)=e_{\sigma\left(i\right)},
\end{equation}
for some $\sigma\in\mathfrak{S}_{3}$ and for $i\in\mathbb{Z}/3\mathbb{Z}$.
Now, considering (\ref{eq:leftaction e0}) we then have that 

\begin{equation}
\mathcal{O}^{\left(i\right)}=\left(e_{i+1}\circ\mathbb{A}_{1/2}\left(\mathcal{O}\right)\right)\cap\left(e_{i+2}\circ\mathbb{A}_{1/2}\left(\mathcal{O}\right)\right),
\end{equation}
where $i\in\mathbb{Z}/3\mathbb{Z}$. Therefore, we have that 
\begin{align}
\Phi\left(\mathcal{O}^{\left(i\right)}\right) & =\Phi\left(\left(e_{i+1}\circ\mathbb{A}_{1/2}\left(\mathcal{O}\right)\right)\cap\left(e_{i+2}\circ\mathbb{A}_{1/2}\left(\mathcal{O}\right)\right)\right)=\\
 & =\Phi\left(e_{i+1}\circ\mathbb{A}_{1/2}\left(\mathcal{O}\right)\right)\cap\Phi\left(e_{i+2}\circ\mathbb{A}_{1/2}\left(\mathcal{O}\right)\right)=\\
 & =\left(\Phi\left(e_{i+1}\right)\circ\mathbb{A}_{1/2}\left(\mathcal{O}\right)\right)\cap\left(\Phi\left(e_{i+2}\right)\circ\mathbb{A}_{1/2}\left(\mathcal{O}\right)\right)=\\
 & =\mathcal{O}^{\left(\tau\left(i\right)\right)},
\end{align}
for some $\tau\in\mathfrak{S}_{3}$ and for every $i\in\mathbb{Z}/3\mathbb{Z}$.
Therefore there exist three automorphisms $\varphi_{0}$, $\varphi_{1}$
and $\varphi_{2}$ of $\mathcal{O}$ such that 
\begin{equation}
\Phi\left(w_{i}\left(x\right)\right)=w_{\tau\left(i\right)}\left(\varphi_{i}\left(x\right)\right).
\end{equation}
A straightforward computation from (\ref{eq:JordanElduque}) shows
that in order for $\Phi$ to be an automorphisms it must be that
\begin{equation}
\varphi_{i}\left(x\right)*\varphi_{i+1}\left(y\right)=\varphi_{i+2}\left(x*y\right),\label{eq:Z3 graded action}
\end{equation}
for every $i\in\mathbb{Z}/3\mathbb{Z}$ and $x,y\in\mathcal{O}$.
Relation (\ref{eq:Z3 graded action}) defines a $\mathbb{Z}/3\mathbb{Z}$
graded action of three copies of $\text{SU}\left(3\right)$ on $\mathbb{A}_{1/2}\left(\mathcal{O}\right)$,
from which we deduce 
\begin{equation}
\text{Aut}\left(\mathbb{A}_{1/2}\left(\mathcal{O}\right)\right)=\mathbb{Z}/3\mathbb{Z}\ltimes\left(\text{SU}\left(3\right)\times\text{SU}\left(3\right)\times\text{SU}\left(3\right)\right),
\end{equation}
where $\mathbb{Z}/3\mathbb{Z}$ is a sort of ``triality symmetry''
which interchanges the $\text{SU}\left(3\right)$ factors in a cyclic
way.

\section{Conclusions and further developments}

In this article we presented a deformation of the Okubic Albert algebra
$\mathbb{A}_{q}\left(\mathcal{O}\right)$ introduced by Elduque, which
for $q=\pm1$ returns a Jordan algebra and for $q=\nicefrac{1}{2}$
constitutes an algebraic equivalent of the Okubic affine and projective
planes. Even though $\mathbb{A}_{\nicefrac{1}{2}}\left(\mathcal{O}\right)$
is not a Jordan algebra, nevertheless it perfectly realizes, for the
Okubic case, the well-known relation between points of the octonionic
plane $\mathbb{O}P^{2}$ and rank-1 idempotent elements of the exceptional
Jordan algebra $\mathfrak{J}_{3}\left(\mathbb{O}\right)$. In the
octonionic case, the automorphisms of $\mathfrak{J}_{3}\left(\mathbb{O}\right)$
yield to an algebro-geometrical realization of the compact real form
of $F_{4\left(-52\right)}$ as the isometry group of the octonionic
projective plane $\mathbb{O}P^{2}$. Analogously, in the Okubic case
we have found that $\text{Aut}\left(\mathbb{A}_{1/2}\left(\mathcal{O}\right)\right)$
is the group $\mathbb{Z}/3\mathbb{Z}\ltimes\text{SU}\left(3\right)^{\times3}$
. A natural development of the present article would be definition
and the study of some sort of projective plane over the split-Okubo
algebra $\mathcal{O}_{s}$. Indeed, even though with some key differences,
similar results to those found for the octonionic projective plane
do hold in the case of the \emph{split-octonionic projective plane}
$\mathbb{O}_{s}P^{2}$ and in the case of the \emph{octonionic hyperbolic
plane} $\mathbb{O}H^{2}$ yielding, when considering the respectively
isometry groups, to concrete realizations of the real form of the
Lie groups $F_{4\left(4\right)}$ and $F_{4\left(-20\right)}$. It
is therefore reasonable to expect that an analogous situation would
hold for a split-Okubic projective plane $\mathcal{O}_{s}P^{2}$ even
though, as in the case of the split-octonions, dealing with a non-division
algebra will yield to non trivial geometrical implications.

\section{Acknowledgment}

The work of AM is supported by a ``Maria Zambrano'' distinguished
researcher fellowship, financed by the European Union within the NextGenerationEU
program.

\bigskip{}
$*$\noun{ Departamento de Matemática, }\\
\noun{Universidade do Algarve, }\\
\noun{Campus de Gambelas, }\\
\noun{8005-139 Faro, Portugal} 
\begin{verbatim}
a55499@ualg.pt.

\end{verbatim}
$\dagger$\noun{ Instituto de Física Teorica, Dep.to de Física,}\\
\noun{Universidad de Murcia, }\\
\noun{Campus de Espinardo, }\\
\noun{E-30100, Spain}
\begin{verbatim}
jazzphyzz@gmail.com
\end{verbatim}
$\ddagger$\noun{ Dipartimento di Scienze Matematiche, Informatiche
e Fisiche, }\\
\noun{Università di Udine, }\\
\noun{Udine, 33100, Italy} 
\begin{verbatim}
francesco.zucconi@uniud.it


\end{verbatim}

\end{document}